\def\year{2019}\relax
\documentclass[letterpaper]{article} 
\usepackage{aaai19}  
\usepackage{times}  
\usepackage{helvet}  
\usepackage{courier}  
\usepackage{url}  
\usepackage{graphicx}  

\usepackage{times}

\usepackage{mathtools}
\usepackage{booktabs}       
\usepackage{amsfonts}       
\usepackage{microtype}      
\usepackage{algorithm}
\usepackage[noend]{algpseudocode}

\usepackage{amsthm}
\usepackage{subcaption}

\newtheoremstyle{mytheoremstyle}{0pt}{0pt}{\itshape}{}{\bf}{.}{.5em}{}
\theoremstyle{mytheoremstyle}

\newtheorem{proposition}{Proposition}
\newtheorem{corollary}{Corollary}
\newtheorem{lemma}{Lemma}
\newtheorem{assumption}{Assumption}
\theoremstyle{definition}

\theoremstyle{remark}

\DeclareMathOperator*{\minimize}{minimize}

\begin{document}
%
\title{NERO: Nested Rebalancing Optimization for Mobility on Demand}
\author{Tomoki Nishi, Satoshi Koide, Keisuke Otaki and Ayano Okoso\\Toyota Central R\&D Labs., Inc. \\Nagakute, Aichi, Japan \\ \{nishi,koide,otaki,okoso\}@mosk.tytlabs.co.jp}

\maketitle
\begin{abstract}
Mobility-on-Demand (MoD) services, such as taxi-like services, are promising applications. Rebalancing the vehicle locations against customer requests is a key challenge in the services because imbalance between the two worsens service quality (e.g., longer waiting times). Previous work would be hard to apply to large-scale MoD services because of the computational complexity. In this study, we develop a scalable approach to optimize rebalancing policy in stages from coarse regions to fine regions hierarchically. We prove that the complexity of our method decreases exponentially with increasing number of layers, while the error is bounded. We numerically confirmed that the method reduces computational time by increasing layers with a little extra travel time using a real-world taxi trip dataset.
\end{abstract}

\section{Introduction}
\noindent \paragraph{Background and Motivation}
Transport-related social problems (e.g., traffic jam) are expected to worsen owing to recent spurts in urbanization. According to estimates by United Nation~(\citeyear{UN2014}), worldwide, the population living in urban areas will increase to around 60\% by 2050. Mobility as a Service, proposed by Hietanen~(\citeyear{Hietanen2014}), is a promising concept for mitigating severe social problems related to transportation. However, Gehrke (\citeyear{Gehrke2018}) reported that Mobility-on-Demand (MoD) services, such as taxi-like services, might worsen the problem because many users of public transportation with large capacity (e.g., buses), will shift to more convenient vehicles with smaller capacity for MoD services. More efficient MoD services are essential to cope with such problems.

A key challenge associated with MoD services is rebalancing idle vehicles to service demand. Some studies have reported that rebalancing considerably improves the rate of requests serviced and customer waiting times~\cite{Spieser2016,Alonso-Mora2017b,Pavone2011}. Iglesias et al.~(\citeyear{Iglesias2018}) formulate the optimal rebalancing problem (\emph{ORP}) as an integer linear programming (ILP) problem using a time-varying network flow model. Their formulation can be solved efficiently due to totally unimodularity of the problem. However, it would be still difficult to incorporate the approach into city-wide services covered a large area with many regions because the computational complexity increases in proportion to the 3.5th power of the number of regions, even when a widely-known LP solver is used~\cite{Vaidya}. One might wonder if the optimization using a coarse mesh size is sufficient to compute the solution quickly; however, we argue that this is not true. In real MoD services, a coarse mesh size declines the service utility because, with such a coarser mesh, the distance from a user and the dispatched vehicle gets longer, leading to longer waiting time.

\paragraph{Statement of contributions} This study aims to develop a scalable method that enables us to rebalance vehicle locations against customer request locations across a large number of regions in MoD services. We propose an approximated, yet highly-scalable method to optimize the rebalancing policy, which is called NEsted Rebalancing Optimization (\emph{NERO}). The key idea is to hierarchically optimize the policy of finer regions at a lower layer using the policy of its upper layer, which is optimized just before, as the constraints.

We determine the computational complexity of NERO for $K$ layers under certain assumptions. To this end, we use the Vaidya method as an LP solver and consider that each region at any layer has constant $M$ nested regions (i.e., children) at its lower layer. We derive the complexity of our method is $M^{3.5(K-1)}$ times smaller than that of a single-layered method.

Moreover, we compute the upper bound of the error in our approach against the single-layered one. We prove that the increase per trip is at most $\frac{M}{(M-1)}\tau$, where $\tau$ is the time to travel the mesh size of regions at the top layer. We also demonstrate that our method with three layers achieves the 98\% shorter computational time than the single-layered method by using a little extra travel time for the rebalancing in numerical experiments with a real-world taxi trip dataset.

\section{Related Work}

Recently, several methods have been developed for solving ORPs. Pavone et al.~(\citeyear{Pavone2011}) formulated an ORP by using a fluid model. They showed that the rebalancing policy could be computed as a solution to an LP problem. Zhang and Pavone~(\citeyear{Zhang2016}) developed a discrete-time model with model predictive control. Iglesias et al.~(\citeyear{Iglesias2018}) formulated the problem as an ILP problem with a time-expanded network. The computational complexities of these approaches increase in a polynomial fashion with the number of regions but is constant with respect to the number of customers and fleet size. Therefore, the incorporation of these methods into real city-wide services that cover many regions would be difficult.

ORPs~\cite{Pavone2011} are related to the following problems: \textit{dynamic trip-vehicle assignment problem}~\cite{Bei2018,Dickerson2018,Pelzer2015a} and \textit{dynamic pickup and delivery problem}~\cite{Berbeglia2010,Parragh2008}. The objective of the problems is to find the optimal vehicle assignment given a set of demands. The ORPs optimize the number of vehicles that travels from one region to another to satisfy demand as opposed to finding the optimal assignment of each vehicle to a request. Therefore, the computational complexity of the methods used to solve the individual problems grows exponentially with increasing number of vehicles and demands, while the complexity of the methods used to solve the ORP remains constant, regardless of the number of vehicles and demands.

Ghosh et al.~(\citeyear{Ghosh2017}) formulated a MIP problem to optimize a profit-maximizing repositioning and routing in bike sharing systems, called Dynamic Repositioning and Routing Problem (DRRP). The method solves the abstract DRRP, where areas are grouped, and then fixes the solution for the original DRRP. Our method can be a generalization of their method, which abstracts stations once, to the method with the multi-layer abstraction. Moreover, we provide the complexity and an upper bound of the extra travel time for the rebalancing theoretically.

A formulation as a spatio-temporal matching problem between demands and suppliers has developed~\cite{Lowalekar2018}. The method also has reduced the computational complexity by abstracting suppliers to zones while the computational complexity of most matching algorithm increases in proportion to the number of the suppliers and the customers. The complexity, however, still increases in proportion to the number of demands due to optimal assignment between the demands and the zones. Moreover, their method cannot optimize the assignment against the long horizon because of matching between requests and suppliers at current and next steps. Our method optimizes the number of vehicles from origins to destinations during multi-time steps.

\section{Preliminaries}

In this section, we introduce the optimal rebalancing problem proposed by Iglesias~\citeyear{Iglesias2018}, which is a problem to optimize the number of vehicle trips to minimize travel cost for the rebalancing. To fit our purpose, we slightly modify their original formulation. The difference and the reason are explained later. Table~\ref{tab:notation} summarizes notation used in this paper.

\subsection{Optimal Rebalancing Problems (ORPs)}

\begin{table}[t]
    \centering
    \caption{\small Notation} \label{tab:notation}
    \small
    \begin{tabular}{ll} \toprule
    \textbf{Symbol} & \textbf{Description} \\ \midrule \midrule
    $\mathcal{N}$       & A set of the $N$ target regions\\
    $\mathcal{N}_{l}$   & A set of regions within region $l$ in its lower layer \\
    $\tau_{ij},\, c_{ij}$ & Travel time and cost between two region, $i$ and $j$ \\
    $\lambda_{ijt}$ & \# of requests from $i$ to $j$ departing at $t\in\mathcal{T}$ \\
    $s_{it}$ & \# of available vehicles in $i$ at $t$ \\
    $x_{ijt}^p$ & \# of transporting vehicles from $i$ to $j$ departing at $t\in\mathcal{T}$\\
    $x_{ijt}^r$ & \# of rebalancing vehicles from $i$ to $j$ departing at  $t\in\mathcal{T}$\\
    $x_{it}^a$ & \# of vehicles arriving into $i$ at $t\in\mathcal{T}$ \\
    $x_{it}^d$ & \# of vehicles departing from $i$ at $t\in\mathcal{T}$ \\
    $K$ & \# of layers \\
    $M$ & \# of children of a region at each layer\\
    $T$ & Length of time intervals\\
    $L$ & Bit size of LP operations with certain precision \\
    \bottomrule
    \end{tabular}
    \vspace{-3mm}
\end{table}

We describe a formulation of ORPs with the time-varying network flow model proposed by Iglesias et al.~(\citeyear{Iglesias2018}). We consider an area $\mathcal{N}$, which is segmented into $N$ regions, that is, $\mathcal{N} \coloneqq \{1,\dots,N\}$. We define $\mathcal{T} \coloneqq [1,\dots, T]$ as an ordered list of length $T$ time intervals. We denote the set of travel time and travel cost from region $i \in \mathcal{N}$ to region $j \in \mathcal{N}$ as $\boldsymbol{\tau}\coloneqq \{ \tau_{ij} \in \mathbb{N}\}_{i,j\in \mathcal{N}}$ and $\mathbf{c}\coloneqq \{c_{ij}\}_{i,j\in \mathcal{N}}$. We represent the set of the number of customer requests to travel from region $i \in \mathcal{N}$ to region $j \in \mathcal{N}$ departing at time $t \in \mathcal{T}$ as $\boldsymbol{\lambda}\coloneqq \{\lambda_{ijt}\}_{i,j\in \mathcal{N},t\in \mathcal{T}}$. Let $\mathbf{x}^p \coloneqq \{x_{ijt}^p\}_{i,j\in \mathcal{N},t\in \mathcal{T}}$ and $\mathbf{x}^r \coloneqq \{x_{ijt}^r\}_{i,j\in \mathcal{N},t\in \mathcal{T}}$ represent the sets of the number of vehicles transporting customers and of the number of rebalancing vehicles from region $i \in \mathcal{N}$ to region $j \in \mathcal{N}$ departing at $t \in \mathcal{T}$, respectively. Let $\mathbf{s}_0 \coloneqq \{s_{i0}\}$ be the number of vehicles available in region $i \in \mathcal{N}$ at time $t = 0$.

ORPs are formulated as an ILP problem under the assumptions of given customer requests $\boldsymbol{\lambda}$ and free starting positions of the vehicles, as follows.
\begin{subequations}
    \begin{align}
        & \minimize_{\mathbf{x}^p,\mathbf{x}^r,\mathbf{s}_0} \sum_{t\in\mathcal{T}} \sum_{i\in\mathcal{N}} \sum_{j\in\mathcal{N}} c_{ij}x^r_{ijt} , ~~\textrm{subject~to}, \label{eq:obj-2} \\
        &x^p_{ijt} = \lambda_{ijt}, ~ {i,j} \in \mathcal{N}, t \in \mathcal{T}, \label{eq:cst-21}\\
        & \sum_{j \in \mathcal{N}}\left(x^r_{ijt}+x^p_{ijt}-x^r_{jit-\tau_{ji}}-x^p_{jit-\tau_{ji}}\right)\nonumber\\
        &\qquad\qquad\qquad = \begin{cases}
            s_{i0} & t = 1 \\
            0  & t > 1 \\
        \end{cases}
        , ~ \forall i \in \mathcal{N}, t \in \mathcal{T}, \label{eq:cst-23}\\
        & \sum_{i\in\mathcal{N}}s_{i0} = V, \label{eq:cst-24}\\
        & x_{ijt}^r, x_{ijt}^p, s_{i0} \in \mathbb{N}, ~ \forall i,j \in \mathcal{N}, t\in\mathcal{T},
    \end{align}
\end{subequations}
where $V$ denotes fleet size given externally. The values of $x_{ijt-\tau_{ji}}^r$ and $x_{ijt-\tau_{ji}}^p$ should be 0 if $t < \tau_{ji}$ because no vehicles depart before time $t=0$. The constraints~\eqref{eq:cst-21} ensures that all customer demands are served. The constraints~\eqref{eq:cst-23} enforces that the number of vehicles arriving equals the number of vehicles departing in each time interval and each region, and vehicles are inserted in the first time interval only.

Here we note the difference of our formulation from previous studies. In the formulation proposed by Iglesias~(\citeyear{Iglesias2018}), the fleet size is also optimized as a result of the balance between the staying costs $\{c_{ii}\}_{i\in\mathcal{N}}$ and the travel costs $\{c_{ij}\}_{i,j\neq i \in \mathcal{N}}$. The number of available vehicles, however, is usually given by MoD services providers and is not so elastic in a short time horizon. Therefore, in the present work, we add the constraints~\eqref{eq:cst-24} to the formulation by Iglesias et al.~(\citeyear{Iglesias2018}) to fix the fleet size. We refer to the method to solve ORPs by using an LP solver as Single-layered Rebalancing Optimization (\emph{SRO}) method hereinafter.  Their formulation is an offline setting and needs that customer demand is given. They relaxed the assumption using model predictive control and the model to predict future customer demand and updating the policy at fixed interval based on model predictive control. The problem can be solved efficiently with an LP solver because it is straightforwardly proven that the problem is also totally unimodular by the theorem~\cite{ghouila1962}.

\subsection{Computational complexity}

In this section, we derive the computational complexity of SRO. The optimal rebalancing policy is obtained by solving an ORP with an LP solver once.
\begin{lemma} \label{prop:complexity_orp1}
Let $C(m,n)$ be the number of arithmetic operations of an LP solver with $m$ constraints and $n$ decision variables. The computational complexity of the SRO is $O(C(NT+1, N^2T+N))$.
\end{lemma}

Suppose we employ the interior point method proposed by Vaidya~(\citeyear{Vaidya}), which requires $C(m,n)=O(m^{1.5}nL)$ arithmetic operations. Then, we obtain the computational complexity of SRO as follows.
\begin{corollary} \label{coro:complexity_orp1}
The computational complexity of SRO with the Vaidya method is $O(N^{3.5}T^{2.5}L)$, where the parameter $L$ is the bit size required to realize each operation.
\end{corollary}
The corollary represents that the complexity increases in proportion to the 3.5th power of the number of regions. The increase would be critical in large-scale MoD services, in which vehicle scheduling are optimized in real time.

\section{Proposed Method}

In this section, we propose a scalable method, called NERO, for ORPs. NERO optimizes the policies using tree-shaped sets of regions (Fig.~\ref{fig:NERO-overview}). The method computes the policy for coarse regions at the upper layer first and then uses the solution to guide the optimization for the finer regions of its lower layer. Before introducing NERO, we first extend the ORP formulation to treat time-varying fleet size because the fleet size in lower layers depends on the policy optimized in upper layers. Second, we introduce an algorithm to find the rebalancing policy by using NERO.

\subsection{ORP with Time-varying Fleet Size}

\paragraph{Problem formulation}
We consider a set of the regions $\mathcal{N}$, which is a subarea segmented with any manner in the whole area $\mathcal{N}_0$. We introduce two additional sets of decision variables $\mathbf{x}^a\coloneqq \{x_{it}^a\}_{i\in \mathcal{N},t\in \mathcal{T}}$ and $\mathbf{x}^d\coloneqq \{x_{it}^d\}_{i\in \mathcal{N},t\in \mathcal{T}}$ to represent the number of empty vehicle arrivals into region $i$ from outside area of $\mathcal{N}$ at time $t$ and the number of vehicle departures from region $i$ to outside area of $\mathcal{N}$ at time $t$, respectively. We represent the number of vehicles at time $t$, which is given externally, as $V_t$. We formulate the ORP with a time-varying fleet size as follows.
\begin{subequations}
    \begin{align}
        & \minimize_{\mathbf{x}^p,\mathbf{x}^r,\mathbf{x}^a,\mathbf{x}^d} \sum_{t\in\mathcal{T}} \sum_{i\in\mathcal{N}} \sum_{j\in\mathcal{N}} c_{ij}x^r_{ijt} , ~~\textrm{subject~to},\label{eq:obj-3}\\
        & x^p_{ijt} = \lambda_{ijt}, ~ {i,j} \in \mathcal{N}, t \in \mathcal{T}, \label{eq:cst-31}
    \end{align}
    \begin{align}
        & \sum_{j \in \mathcal{N}}\left(x^r_{ijt}+x^p_{ijt}-x^r_{jit-\tau_{ji}}-x^p_{jit-\tau_{ji}}\right) = s_{it},\nonumber\\
        &\qquad\qquad\qquad\qquad\qquad\qquad\quad\forall i \in \mathcal{N}, t \in \mathcal{T}, \label{eq:cst-32}\\
        & s_{it} =
        \begin{cases}
            s_{i0} & t = 1 \\
            x^a_{it}-x^d_{it} & t > 1 \\
        \end{cases}
        , ~\forall i \in \mathcal{N}, t \in \mathcal{T}, \label{eq:cst-33}\\
        & \sum_{i\in\mathcal{N}}\left(s_{it} + \sum_{j\in\mathcal{N}}\sum_{\Delta t=1}^{\tau_{ji}} \left(x^r_{jit-\Delta t}+x^p_{jit-\Delta t}\right)\right) = V_{t}, \nonumber\\
        &\qquad\qquad\qquad\qquad\qquad\qquad\quad\qquad\quad \forall t \in \mathcal{T}, \label{eq:cst-34}\\
        & \sum_{i \in \mathcal{N}} x^a_{it} =
        \begin{cases}
         V_{t}-V_{t-1} & V_{t} \leq V_{t-1} \\
         0  & \textrm{otherwise}
        \end{cases}
        ~,\forall t \in \mathcal{T}, \label{eq:cst-35} \\
        & \sum_{i \in \mathcal{N}} x^d_{it} =
        \begin{cases}
         0 & V_{t} \leq V_{t-1} \\
         V_{t-1}-V_{t}  & \textrm{otherwise}
        \end{cases}
        ~,\forall t \in \mathcal{T}, \label{eq:cst-36} \\
        & x_{ijt}^r, x_{ijt}^p, x_{it}^a, x_{it}^d, s_{i0} \in \mathbb{N}, ~ \forall i,j \in \mathcal{N}, t\in\mathcal{T},
    \end{align}
\end{subequations}
where let $V_0$ be 0. The differences compared to the ORP formulation are as follows.

\begin{figure}[t]
 \centering
 \includegraphics[width=0.7\hsize]{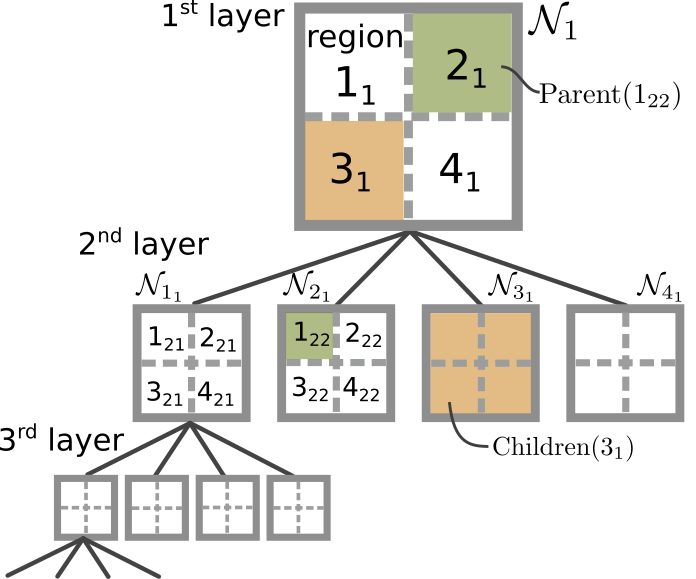}
 \caption{\small Tree-shaped sets of regions for NERO}\label{fig:NERO-overview}
 \vspace{-3 mm}
\end{figure}

\begin{itemize}\setlength{\leftskip}{0cm}\setlength{\itemsep}{0cm}\setlength{\parskip}{0cm}
    \item The conservation constraints of vehicle flow~\eqref{eq:cst-23} is replaced with the constraints~\eqref{eq:cst-32} and \eqref{eq:cst-33} to enforce that the number of vehicles arriving from inside/outside of $\mathcal{N}$ equals the number of vehicles departing the inside/outside regions in each time interval and each region.
    \item The fleet size constraints~\eqref{eq:cst-24} is replaced with the constraints~\eqref{eq:cst-34} because the size changes dynamically. The total number of vehicles within $\mathcal{N}$ can be calculated by summing up the numbers of vehicles idling at region $i \in \mathcal{N}$ and vehicles heading to region $i$.
    \item The constraints~\eqref{eq:cst-35} and ~\eqref{eq:cst-35} are added to satisfy the conservation law of the number of vehicles in the parent region of the upper layer when using the NERO method.
\end{itemize}

The problem can also be solved efficiently with an LP solver because it is straightforwardly proven that the problem is also totally unimodular by the theorem~\cite{ghouila1962}. The ORPs with a time-varying fleet size are also offline settings and need customer demand but can be extended online settings with predicted future demand by following the previous approach~\cite{Iglesias2018}.

\paragraph{Computational complexity}
We derived the computational complexity of the ORPs with time-varying fleet size using Vaidya method, which has $NT+2T$ constraints and $N^2T+2NT$ decision variables, as follows.
\begin{corollary}
The computational complexity of an ORP with time-varying fleet size using Vaidya method is $O(N^{3.5}T^{2.5}L)$.
\end{corollary}

\subsection{NERO: Nested Rebalancing Optimization} \label{subsec:NERO}

We introduce a scalable algorithm to find the rebalancing policy for ORPs. First, we aggregate a given set of regions in a hierarchical manner and obtain tree-shaped coarser sets of regions. After optimizing the policy for trips across coarser regions in the upper layers, NERO optimizes the policy for trips across finer regions in the lower layers using the number of vehicles in its parent region as the constraints.

We define two operators for region $l$ in a layer: $\textrm{Children}(l)$ and $\textrm{Parent}(l)$. Children$(l)$ returns $\mathcal{N}_{l}$, which is a set of regions segmented region $l$ into finer regions. Parent$(l)$ returns a parent region of region $l$ in its upper layer. The sets of travel time, travel cost, and the number of customer requests of regions in $\mathcal{N}_l$ are denoted by $\boldsymbol{\tau}_{l} \coloneqq \{\tau_{ijt}\}_{i,j\in \mathcal{N}_{l},t\in \mathcal{T}}$, $\mathbf{c}_{l}\coloneqq \{c_{ijt}\}_{i,j\in \mathcal{N}_{l},t\in \mathcal{T}}$ and $\boldsymbol{\lambda}_{l}\coloneqq \{\lambda_{ijt}\}_{i,j\in \mathcal{N}_{l},t\in \mathcal{T}_{l}}$, respectively. We denote the rebalancing policy in $\mathcal{N}_l$ as $\mathbf{x}_l\coloneqq (\mathbf{x}_l^p,\mathbf{x}_l^r,\mathbf{x}_l^a,\mathbf{x}_l^d)$, where $\mathbf{x}_l^p \coloneqq \{x_{ijt}^p\}_{i,j\in \mathcal{N}_l,t\in \mathcal{T}}$,  $\mathbf{x}_l^r \coloneqq \{x_{ijt}^r\}_{i,j\in \mathcal{N}_l,t\in \mathcal{T}}$, $\mathbf{x}_l^a \coloneqq \{x_{it}^a\}_{i\in \mathcal{N}_l,t\in \mathcal{T}}$ and $\mathbf{x}_l^a \coloneqq \{x_{it}^d\}_{i\in \mathcal{N}_l,t\in \mathcal{T}}$. We represent an ordered list of the fleet size in region $l$ in the time interval $\mathcal{T}$ as $\mathbf{V}_l\coloneqq [V_{lt}]_{t\in \mathcal{T}}$, which is used to find the rebalancing policy within region $l$ and calculated using the rebalancing policy optimized within the parent of the region $l$, as follows.
\begin{align}
    & V_{lt} = s_{lt} + \sum_{j\in\mathcal{N}} \left(x^r_{jlt-\tau_{jl}}+x^p_{jlt-\tau_{jl}}\right), t \in \mathcal{T}, \label{eq:fleetsize}\\
    & s_{lt} =
    \begin{cases}
        s_{l0} & t = 1 \\
        x^a_{lt}-x^d_{lt} & t > 1 \\
    \end{cases}
    , t \in \mathcal{T}.\nonumber
\end{align}
We introduce two operators for the set of regions $\mathcal{N}_l$. Demand($\mathcal{N}_l$) returns the number of customer requests $\boldsymbol{\lambda}_l$ within $\mathcal{N}_l$. Rebalancing($\mathcal{N}_l,\boldsymbol{\tau}_l,\mathbf{c}_l,\mathbf{V}_l,\boldsymbol{\lambda}_l$) returns the optimal rebalancing policy $\mathbf{x}_l$ through the ORP with varying fleet size. The total number of vehicles is determined by the number of vehicles in the parent region. We represent an ordered list of the total fleet size at each time as $\mathbf{V}_{max}$.

We show the nested algorithm with $K$ layers, called NERO$_K$, in Algorithm~\ref{alg:nero}. Our proposed method finds the rebalancing policy by calling NERO$_K(\mathcal{N}_1,1)$, where $\mathcal{N}_1$ is a set of regions at the top layer. At line 7, the policies at each layer are recursively optimized using the fleet size and the demand, which are respectively computed in line 4 or 5, and line 6. It is straightforward to convert the set of the rebalancing policies $\mathbf{x} \coloneqq \{\mathbf{x}_l\}_{l\in \mathcal{N}_l}$ returned by NERO$_K(\mathcal{N}_1,1)$ to the rebalancing policy for the ORP at the finest mesh size.

\begin{algorithm}[t]
\small
\caption{\small $\textrm{NERO}_K(\mathcal{N},k)$}\label{alg:nero}
\begin{algorithmic}[1]
\Require{set of region $\mathcal{N}$, index of layer $k$}
\Ensure{set of rebalancing policies $\mathbf{x}$}
\State{$\mathbf{x} \leftarrow \emptyset $}
\For {$l$ in $\mathcal{N}$}
    \State{$\mathcal{N}_l \leftarrow \textrm{Children}(l)$}
    \State{\algorithmicif\ $k=1$\ \algorithmicthen\ $\mathbf{V}_l \leftarrow \mathbf{V}_{max}$\ }
    \State{\algorithmicelse\ $\mathbf{V}_l \leftarrow \textrm{Computed using Eq.~\eqref{eq:fleetsize}}$}
    \State{$\boldsymbol{\lambda}_l \leftarrow \textrm{Demand}(\mathcal{N}_l)$}
    \State{ $(\mathbf{x}^r_l, \mathbf{x}^p_l, \mathbf{x}^a_l, \mathbf{x}^d_l) \leftarrow \textrm{Rebalancing}(\mathcal{N}_l,\boldsymbol{\tau}_l,\mathbf{c}_l,\mathbf{V}_l,\boldsymbol{\lambda}_l)$}
    \State{$\mathbf{x} \leftarrow \mathbf{x} \cup \{(\mathbf{x}^r_l, \mathbf{x}^p_l, \mathbf{x}^a_l, \mathbf{x}^d_l)\}$}
    \State{\algorithmicif\ $k \leq K$\ \algorithmicthen\ $\mathbf{x} \leftarrow \mathbf{x} \cup \textrm{NERO}_K(\mathcal{N}_l,k+1)$ }
\EndFor
\State{$\textrm{return}~ \mathbf{x}$}
\end{algorithmic}
\vspace{-1mm}
\end{algorithm}

\paragraph{Computational complexity}
We discuss the computational complexity of NERO$_K$. We introduce the following assumption to make the discussion easy, albeit NERO can handle any segmented region:
\begin{assumption}\label{asmpt:segmentation}
The number of children of any region is $M$.
\end{assumption}
The assumption means that each region is segmented into $M$ regions in its lower layers, except for the top layer. We introduce the following proposition and two corollaries for determining the complexity of NERO under the assumption.
\begin{proposition}\label{prop:complexity_NERO1}
Consider the number of regions in the bottom layer to be $N$. The computational complexity of NERO$_K$ is
\begin{align*}
& O\left(C\left(\frac{NT}{M^{K-1}}+2T, \frac{N^2T}{M^{2(K-1)}}+\frac{2NT}{M^{K-1}}\right)\right.\\
&\qquad\quad \left.+~C\left(MT+2T,M^2T+2MT\right)\frac{NM(M^{K-1}-1)}{M^{K-1}(M-1)}\right).
\end{align*}
\end{proposition}
\begin{proof}
The number of the regions in the top layer is $N/M^{K-1}$ because of Assumption~\ref{asmpt:segmentation}. Thus, the computational complexity of optimization in the top layer is
\begin{align*}
O\left(C\left(\frac{N}{M^{K-1}}+2T, \frac{N^2T}{M^{2(K-1)}}+\frac{2NT}{M^{K-1}}\right)\right).
\end{align*}
The number of regions in layer $k$ is $N/M^{K-k}$, and the number of regions within each region in the layer is $M$ according to Assumption~\ref{asmpt:segmentation}. Thus, the computational complexity of optimization in layer $k$ is $O(C_{M})N/M^{K-k}$, where $C_{M}\coloneqq C(MT+2T, M^2T+2MT)$. Summing up the computational complexity from $k=2$ to $K$ is
\begin{align*}
\sum_{k=2}^{K} \frac{O(C_M)N}{M^{K-k}} & = \frac{NM(M^{K-1}-1)}{M^{K-1}(M-1)}O(C_M).
\end{align*}
Because the computational complexity of NERO$_K$ is the sum of the top layer complexity and the total complexity, the proposition is derived.
\end{proof}

We can naturally derive the following corollary for the convergence of the computational complexity.
\begin{corollary}\label{prop:complexity_NERO3}
The computational complexity of NERO$_K$ converges to $O\left(C\left(MT+2T,M^2T+2MT\right)\frac{NM}{(M-1)}\right)$ with increasing the number of layers.
\end{corollary}

We also derives the computational complexity of NERO$_k$ with Vaidya method by using $C(m,n)=m^{1.5}nL$.
\begin{corollary}\label{prop:complexity_NERO2}
The computational complexity of NERO$_K$ with the Vaidya method is $O\left(\frac{N^{3.5}T^{2.5}L}{M^{3.5(K-1)}}\right)$ for $N$ in the bottom layer and is $O\left(\frac{NM^{4.5}T^{2.5}L}{(M-1)}\right)$ for $M$.
\end{corollary}
According to the proposition, the complexity of NERO$_k$ decreases exponentially as the number of layers increases and $M^{3.5(K-1)}$ times lower than that of SRO. Therefore, NERO can be used to obtain the rebalancing policy even for large-scale MoD services by setting a suitable number of layers.

\paragraph{Additional rebalancing time of NERO}

Vehicles optimized by NERO need more travel time for rebalancing, called the \emph{rebalancing time} hereafter, than ones by SRO because the rebalancing is optimized using longer travel time to trip to other areas at upper layers. Figure~\ref{fig:travel-distance} shows the difference in the travel distance used in optimization between SRO and NERO. The red arrow and the blue arrow, respectively, denote a trip to the center point within regions and a trip across regions. The rebalancing policy using NERO estimates extra travel time to trip to the customer. The apparent extra travel distance is reflected as the additional waiting time.

We introduce the following assumption to make the discussion easy, albeit NERO can handle any segmented region:
\begin{assumption}\label{asmpt:region_size}
Regions in each layer are segmented with the same size mesh uniformly.
\end{assumption}
We show the following proposition related to the increase owing to the use of NERO under Assumptions~\ref{asmpt:segmentation} and \ref{asmpt:region_size}.
\begin{proposition}\label{prop:travel_time_loss1}
Let $\tau$ be the time to travel the mesh size distance of a region at the top layer. Then, the rebalancing time per vehicle increases at most by $\frac{M}{(M-1)}\tau$.
\end{proposition}
\begin{proof}
We first discuss the case when the assignment between vehicles and their destinations is not changed by NERO from SRO. At the top layer, a vehicle travels for the rebalancing at most $\tau$ longer. A vehicle may lose $\tau/M$ at the second layer and do $\tau/M^2$ at the third layer. Generally, we may lose $\tau/M^{(k-1)}$ at $k$-th layer, where $k>1$. Summing up these time over $k$, we obtain the total travel time as $\tau+\tau\cdot\sum_{k=1}^\infty 1/M^k=\tau+\tau\cdot(\frac{M}{M-1}-1)=\frac{M}{M-1}\tau$. Even if the assignment is changed, the upper bound is satisfied because the identical assignment worsen the travel time compared to the changed assignment and $\frac{M}{M-1}\tau$ bounds its travel time from the top.
\end{proof}
The proposition claims that the upper bound of the increase in the rebalancing time optimized by NERO is in proportion to the mesh size of the top layer. 

\begin{figure}[t]
 \centering
  \includegraphics[width=1.0\hsize]{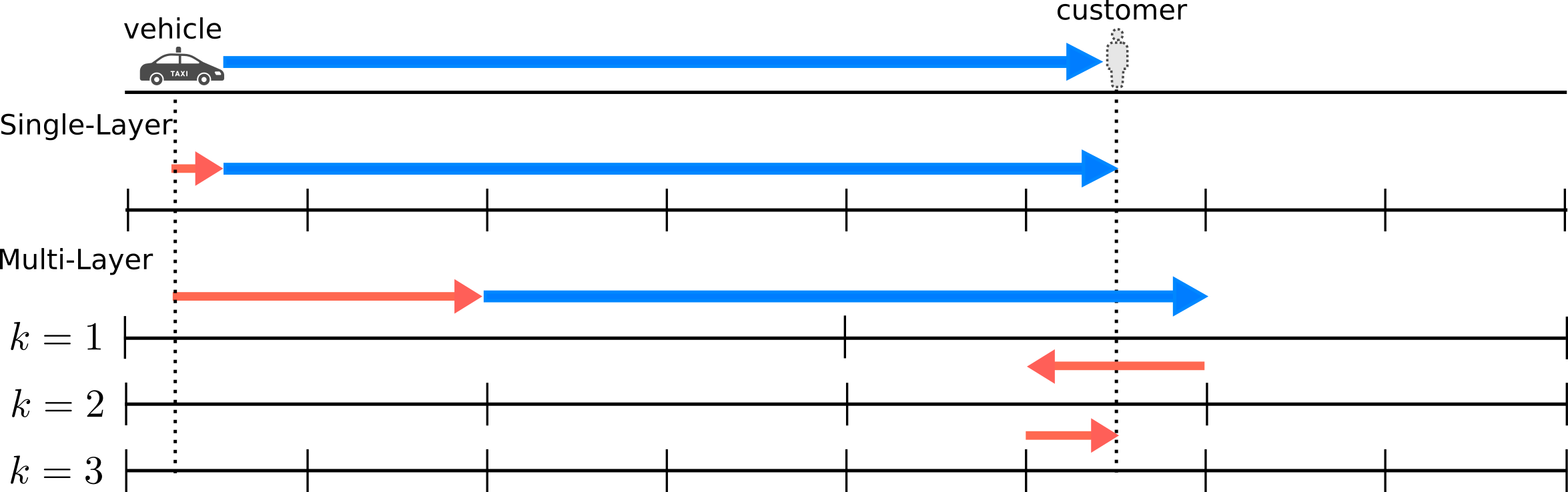}
 \caption{\small Policies optimized by SRO and NERO with three layers. The red arrow and the blue arrow, respectively, denote travel to the center point within regions and travel across regions.}\label{fig:travel-distance}
 \vspace{-4mm}
\end{figure}

\section{Experiments}

We evaluate NERO by using an open dataset of taxi trips in Manhattan, New York City, USA~\cite{nytaxi}.

\subsection{Dataset}

The dataset used herein consisted of the data of 230,620 trips requested by customers over 24 h on May 3, 2016. For each trip, the dataset contains locations (latitude and longitude) of origin and destination, and time of a request. The maximum and the minimum number of customer requests were 16,000 at 8 pm and 1,200 at 3 pm respectively. We used five mesh sizes to divide the area uniformly into regions for optimization: 250-m, 500-m, 1-km, 2-km and 4-km mesh. We obtained graphs of the segmented regions based on the road network of Manhattan from Openstreetmap~\cite{osm} because vehicles might not directly travel from a region to next regions in the mesh due to some reasons such as a park.

\subsection{Experimental Setup}

We optimized the rebalancing at intervals of an hour independently and evaluated the mean of the computational time and the rebalancing time per vehicle of 24 experiments. We assumed the average vehicle speed $v_{avg}$ as 5.5 m/s and thus the travel time to an adjacent region $\tau$ is calculated by $\Delta/v_{avg}$, where $\Delta$ is the mesh size. For example, $\tau$ is three minutes in case of the 500-m mesh. The travel time $\tau_{ij}$ was also calculated with $D_{ij}/v_{avg}$, where $D_{ij}$ is $L_1$-distance from region $i$ to region $j$.

We used $\tau$ of the bottom layer as the length of a time step for ORPs, that is, the number of the time steps $T$ in $\mathcal{T}$ is $60/\tau$ because the length of the intervals was an hour. For example, $T$ was 40 when using 250-m mesh at the bottom layer. The number of regions $N$ in the mesh was 261. We set $\tau_{ii}$ to one. The rebalancing costs $c_{ij}$ were found to be proportional to the travel time $\tau_{ij}$, except for $c_{ii}$, which was set to zero. $N$ and $\tau$ of the other meshes are listed in Table~\ref{tab:parameters}. We calculated NERO on a PC equipped with an Intel Xeon CPU Broadwell@2.6 GHz and 112 GB memory. We optimized the rebalancing policy by calling an LP solver included in Gurobi-8.5~(\citeyear{Gurobi}) with default settings from a Python code.

\begin{table}[htbp]
 \centering
 \caption{\small Number of regions and travel time to an adjacent region} \label{tab:parameters}
 \begin{tabular}{cccccc}
     \toprule
     & 250-m mesh& 500 m & 1 km & 2 km & 4km\\
     \midrule
     \midrule
    $N$ & 867 & 261 & 79& 26 & 10\\
    $\tau$ & 1.5 min. & 3 & 6 & 12 & 24\\
    \bottomrule
\end{tabular}
\vspace{-2mm}
\end{table}

In the experiments, we compared NERO to SRO for the computational time and the rebalancing time but not the customer waiting time because of the followings. When the minimum mesh sizes are common in the methods, NERO always realizes the same waiting time as SRO due to the constraints~\eqref{eq:cst-21} or the constraints~\eqref{eq:cst-31}. Iglesias et al.~(\citeyear{Iglesias2018}) already shown that the waiting time of SRO was much shorter than the other conventional methods (e.g., Pavone et al.~(\citeyear{Pavone2011})) in their previous work. Therefore,  we only need to evaluate NERO for the computational time and the rebalancing time compared to SRO because NERO also obviously achieves the same or much shorter waiting time than the previous methods.

We compared the NEROs with several layers to SRO. We evaluated our method with two minimum mesh sizes: 250 m and 500 m. Table~\ref{tab:method} shows the settings of the methods. We set the length of each time step to 3 minutes when the minimum mesh size was 500 m and to 1.5 minutes when it was 250 m. NERO$_l^m$ means that NERO contains $l$ layers and its minimum mesh size is $m$. SRO$^{m}$ is SRO with a mesh size of $m$.

Fleet size is necessary to optimize the rebalancing policy. We decided fleet size based on the dataset. We used a fleet size contained 110\% of the number of unique transported vehicles in each time interval. The maximum and minimum fleet size were 5,900 at 7 pm and 1,100 at 3 am respectively.

\begin{table}[htbp]
 \centering
 \caption{\small Settings of methods compared herein} \label{tab:method}
 \small
 \begin{tabular}{c|ccccc}
     \toprule
     method& 1st-layer & 2nd & 3rd & 4th & 5th\\
     \midrule
     \midrule
    SRO$^{500}$ & 500-m mesh & & & &\\
    NERO$_2^{500}$ &  1 km & 500 m & & &\\
    NERO$_3^{500}$ &  2 km& 1 km & 500 m & & \\
    NERO$_4^{500}$ &  4km & 2 km& 1 km & 500 m & \\
     \midrule
    SRO$^{250}$ &  250-m & & & &\\
    NERO$_2^{250}$  & 500 m & 250 m & & &\\
    NERO$_3^{250}$  & 1 km& 500 m & 250 m  & &\\
    NERO$_4^{250}$  & 2km & 1 km& 500 m & 250 m  &\\
    NERO$_5^{250}$  & 4km & 2km &1 km& 500 m & 250 m\\
    \bottomrule
\end{tabular}
\vspace{-4mm}
\end{table}

\subsection{Results}
\paragraph{Computational complexity}

We evaluated the computational time with respect to the number of layers. We could not receive any feasible solutions when setting the time limit to the computational time of NERO$_4^{250}$ (i.e., 10.4 seconds) to optimize SRO. The total computational time and the optimization time of each method are shown in Fig.~\ref{fig:result1} (a): the blue line and the red dash-line represent the computational time of the methods with 250-m mesh and of the one with 500-m mesh as the minimum mesh. NERO$_3^{500}$ shortened the computational time by 158 seconds compared to SRO$^{500}$, and NERO$_4^{250}$ did the time by more than 5,400 seconds. That is, both NERO$_4^{250}$ and NERO$_4^{500}$ decreased the computational time by more than 98\% compared to SRO$^{250}$ and SRO$^{500}$, respectively.

We confirmed that the computational time decreased exponentially until three layers and the decrease was saturated after four layers, as expected in Proposition~\ref{prop:complexity_NERO1} and \ref{prop:complexity_NERO3}. The slight increase after three layers is also explainable by discussing which terms of the complexity in Proposition~\ref{prop:complexity_NERO1} are dominant. That is, the complexity would decrease exponentially with the number of layers when the first term is dominant. On the other hand, the complexity would increase and then saturate with increasing the number of layers when the second term is dominant.

\paragraph{Rebalancing time}
We evaluated the rebalancing time per vehicle. The rebalancing time ratio for the total service time, which is 60 minutes in the experiments, of each method is shown in Fig.~\ref{fig:result1} (b). The ratio was computed using $\sum_{k=0}^K\sum_{\{N_l\}_k} \sum_{i,j \in \mathcal{N}_l,t\in \mathcal{T}} (x^r_{ijt}\tau_{ij} + x^a_{it}\tau_k/2)/60$, where $\{N_l\}_k$ is the set of the sets of regions at layer $k$ and $\tau_k$ is the mesh size of a region at layer $k$. $\tau_k/2$ means the travel time from a border to the center in a region at layer $k$. We confirmed that the rebalancing time monotonically increases for the mesh size at the top layer, as expected according to Proposition~\ref{prop:travel_time_loss1}. For example, the policy optimized by NERO$_3^{250}$ spent 5.4~$\%$ for the rebalancing in the service time (i.e., three minutes in an hour) while the policy using SRO$^{250}$ did 0.4~$\%$. However, we found that the increase is much slower than the linear with respect to the mesh size at the top layer. This would be because most trips, in reality, are the much shorter distance than the assumption of Proposition~\ref{prop:travel_time_loss1}, where all trips travel across the regions at the top layer.

\begin{figure}[t]
 \centering
  \begin{subfigure}{1.0\hsize}
  \centering
  \includegraphics[width=1.0\hsize]{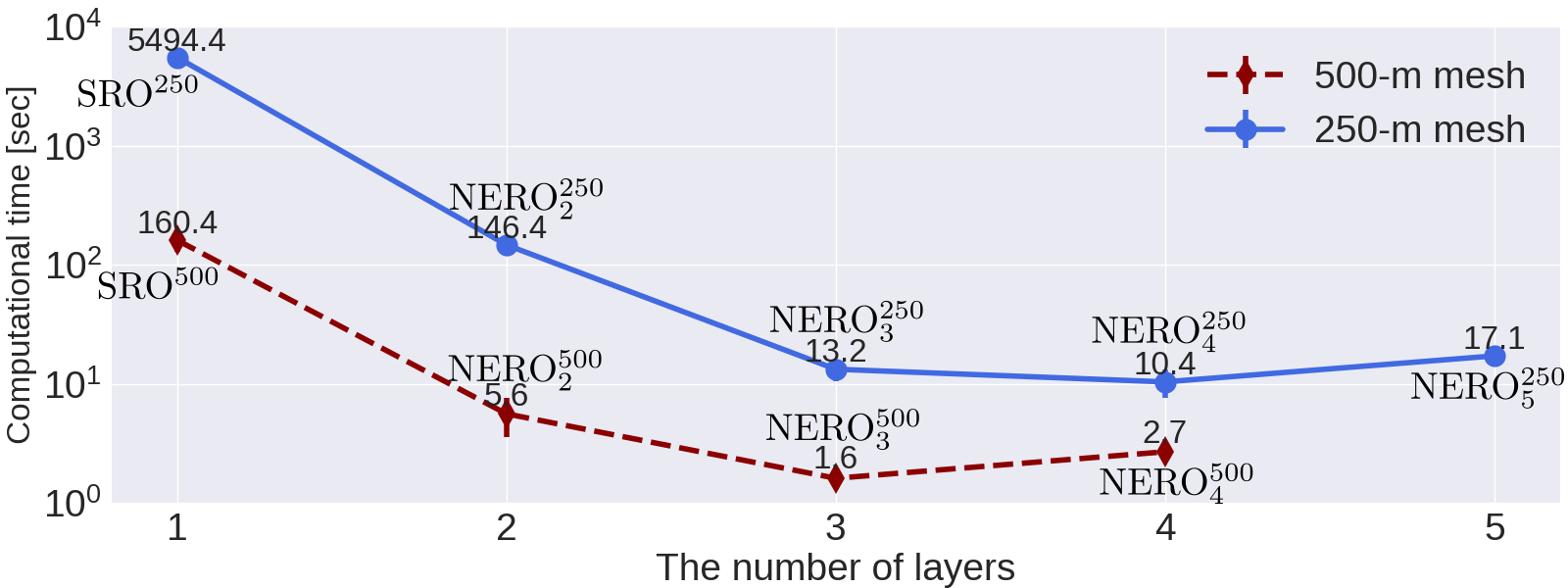}
  \caption{\small Computational time}
 \end{subfigure}
 \begin{subfigure}{1.0\hsize}
  \centering
  \includegraphics[width=1.0\hsize]{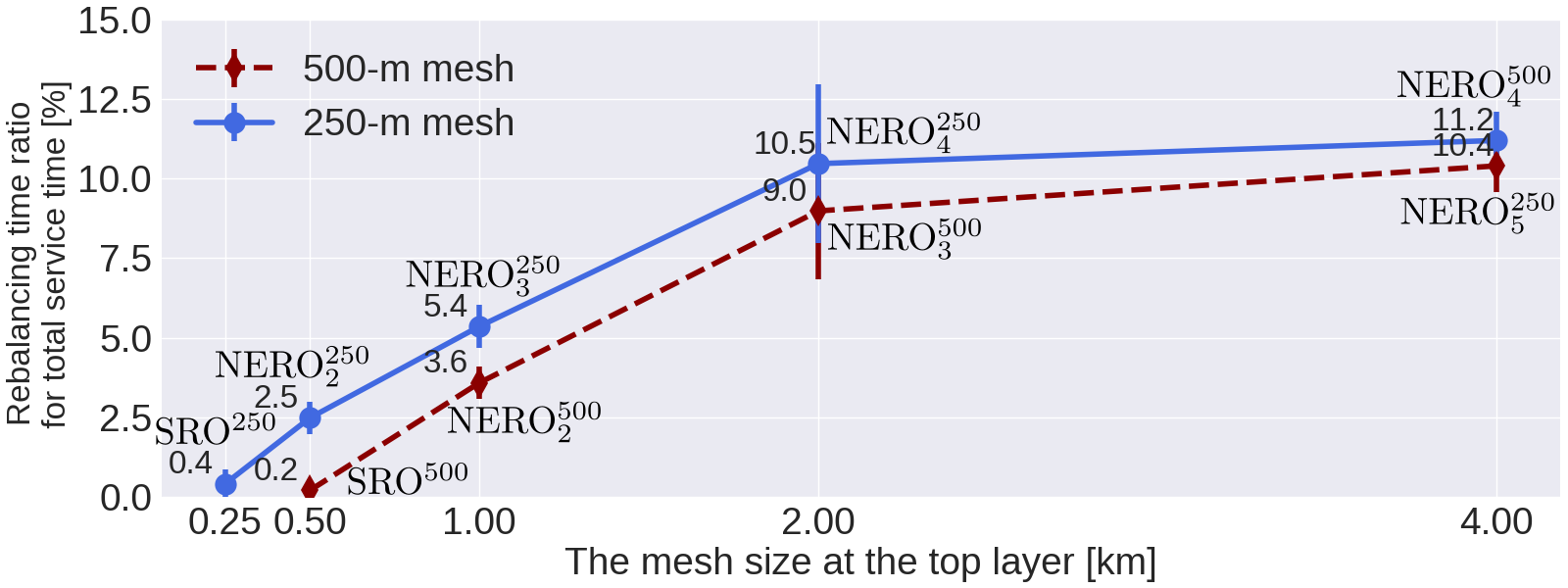}
  \caption{\small Rebalancing time ratio for the total service time}
 \end{subfigure}
 \vspace{-2mm}
\caption{\small Computational time (a) and rebalancing time ratio for the total service time per vehicle (b). The blue line and the red dash-line are the values of the methods with 250-m mesh and 500-m mesh as the minimum mesh respectively. The error bars denote the standard deviation of the values. }\label{fig:result1}
\end{figure}

\section{Concluding remarks}

In this paper, we presented a scalable approach to optimize vehicle scheduling in large-scale MoD services with a number of regions. We first developed the time-expanded network flow model with a predetermined time-varying fleet size: the model allowed us to change the total number of vehicles with the time, whereas the conventional model constrains us to fix the total number. We presented an algorithm called NERO to optimize vehicle scheduling and routing in stages from coarse regions to fine regions by using structured regions hierarchically. We also theoretically analyzed the computational complexity and the extra travel time for the rebalancing of the proposed method under an assumption about region segmentation. We proved that the computational complexity decreased exponentially with increasing number of layers. We also derived the increase in the travel time of NERO in proportion to the mesh size at the top layer theoretically.

We numerically evaluated our algorithm using a real-world taxi trip dataset. We confirmed that our hierarchical algorithm can compute the solution of ORPs significantly faster than the single layer method by spending the additional rebalancing time. For instance, NERO$^{250}_{3}$ achieved 98$\%$ reduction of the computational time by using just 5.4 $\%$ in the service time as the extra rebalancing time. Our method can balance between the computational time and the extra rebalancing time by tuning the number of layers. As a future work, we will consider extending NERO to handle predicted demand in travel times.
\bibliographystyle{aaai}
\bibliography{references}

\begin{thebibliography}{}

\bibitem[\protect\citeauthoryear{Alonso-Mora, Wallar, and
  Rus}{2017}]{Alonso-Mora2017b}
Alonso-Mora, J.; Wallar, A.; and Rus, D.
\newblock 2017.
\newblock {Predictive Routing for Autonomous Mobility-on-Demand Systems with
  Ride-Sharing}.
\newblock In {\em the IEEE/RSJ Conf. on Robotics and Intelligent Systems
  (IROS)}.

\bibitem[\protect\citeauthoryear{Bei and Zhang}{2018}]{Bei2018}
Bei, X., and Zhang, S.
\newblock 2018.
\newblock {Algorithms for Trip-Vehicle Assignment in Ride-Sharing}.
\newblock In {\em AAAI Conference on Artificial Intelligence}.

\bibitem[\protect\citeauthoryear{Berbeglia, Cordeau, and
  Laporte}{2010}]{Berbeglia2010}
Berbeglia, G.; Cordeau, J.~F.; and Laporte, G.
\newblock 2010.
\newblock {Dynamic pickup and delivery problems}.
\newblock {\em European Journal of Operational Research} 202(1):8--15.

\bibitem[\protect\citeauthoryear{Brian and Dan}{2016}]{nytaxi}
Brian, D., and Dan, W.
\newblock 2016.
\newblock New york city taxi trip data (2010-2013).

\bibitem[\protect\citeauthoryear{Dickerson \bgroup et al\mbox.\egroup
  }{2018}]{Dickerson2018}
Dickerson, J.~P.; Sankararaman, K.~A.; Srinivasan, A.; and Xu, P.
\newblock 2018.
\newblock {Allocation Problems in Ride-Sharing Platforms: Online Matching with
  Offline Reusable Resources}.
\newblock In {\em AAAI Conference on Artificial Intelligence}.

\bibitem[\protect\citeauthoryear{Gehrke}{2018}]{Gehrke2018}
Gehrke, S.
\newblock 2018.
\newblock {A Survey of Ride-Hailing Passengers}.
\newblock In {\em TREC Friday Seminar Series. 152.}

\bibitem[\protect\citeauthoryear{Ghosh \bgroup et al\mbox.\egroup
  }{2017}]{Ghosh2017}
Ghosh, S.; Varakantham, P.; Adulyasak, Y.; and Jaillet, P.
\newblock 2017.
\newblock {Dynamic repositioning to reduce lost demand in Bike Sharing
  Systems}.
\newblock Technical report.

\bibitem[\protect\citeauthoryear{Ghouila-Houri}{1962}]{ghouila1962}
Ghouila-Houri, A.
\newblock 1962.
\newblock Charact{\'e}risation de matrices totalement unimodulaires.
\newblock {\em Comptes Redus Hebdomadaires de S{\'e}ances de l'Acd{\'e}mie des
  Sciences (Paris)} 254:1192--1194.

\bibitem[\protect\citeauthoryear{{Gurobi Optimization, LLC}}{2018}]{Gurobi}
{Gurobi Optimization, LLC}.
\newblock 2018.
\newblock Gurobi optimizer reference manual.

\bibitem[\protect\citeauthoryear{Haklay and Weber}{2008}]{osm}
Haklay, M., and Weber, P.
\newblock 2008.
\newblock Openstreetmap: User-generated street maps.
\newblock {\em IEEE Pervas Comput} 7(4):12--18.

\bibitem[\protect\citeauthoryear{Hietanen}{2014}]{Hietanen2014}
Hietanen, S.
\newblock 2014.
\newblock {Mobility as a Service Can it be even better than owning a car?}
\newblock In {\em The new transport model}.

\bibitem[\protect\citeauthoryear{Iglesias \bgroup et al\mbox.\egroup
  }{2018}]{Iglesias2018}
Iglesias, R.; Rossi, F.; Wang, K.; Hallac, D.; Leskovec, J.; and Pavone, M.
\newblock 2018.
\newblock {Data-Driven Model Predictive Control of Autonomous
  Mobility-on-Demand Systems}.
\newblock In {\em IEEE International Conference on Robotics and Automation
  (ICRA)},  1--7.

\bibitem[\protect\citeauthoryear{Lowalekar, Varakantham, and
  Jaillet}{2018}]{Lowalekar2018}
Lowalekar, M.; Varakantham, P.; and Jaillet, P.
\newblock 2018.
\newblock {Online spatio-temporal matching in stochastic and dynamic domains}.
\newblock {\em Artificial Intelligence}.

\bibitem[\protect\citeauthoryear{Parragh, Doerner, and
  Hartl}{2008}]{Parragh2008}
Parragh, S.~N.; Doerner, K.~F.; and Hartl, R.~F.
\newblock 2008.
\newblock {A survey on pickup and delivery problems}.
\newblock {\em Journal fur Betriebswirtschaft} 58(1):21--51.

\bibitem[\protect\citeauthoryear{Pavone \bgroup et al\mbox.\egroup
  }{2011}]{Pavone2011}
Pavone, M.; Smith, S.~L.; Frazzoli, E.; and Rus, D.
\newblock 2011.
\newblock {Load Balancing for Mobility-on-Demand Systems}.
\newblock In {\em Robotics: Science and Systems VII (RSS)}.

\bibitem[\protect\citeauthoryear{Pelzer \bgroup et al\mbox.\egroup
  }{2015}]{Pelzer2015a}
Pelzer, D.; Xiao, J.; Zehe, D.; Lees, M.~H.; Knoll, A.~C.; and Aydt, H.
\newblock 2015.
\newblock {A Partition-Based Match Making Algorithm for Dynamic Ridesharing}.
\newblock {\em IEEE Transactions on Intelligent Transportation Systems}
  16(5):2587--2598.

\bibitem[\protect\citeauthoryear{Spieser \bgroup et al\mbox.\egroup
  }{2016}]{Spieser2016}
Spieser, K.; Samaranayake, S.; Gruel, W.; and Frazzoli, E.
\newblock 2016.
\newblock {Shared-vehicle Mobility-on-Ddemand Systems: A Fleet Operator's Guide
  to Rebalancing Empty Vehicles}.
\newblock {\em TRB 2016 Annual Meeting}  0--16.

\bibitem[\protect\citeauthoryear{UN}{2014}]{UN2014}
UN.
\newblock 2014.
\newblock World urbanization prospects: The 2014 revision population database.
\newblock Technical report.

\bibitem[\protect\citeauthoryear{Vaidya}{1989}]{Vaidya}
Vaidya, P.
\newblock 1989.
\newblock {Speeding-up linear programming using fast matrix multiplication}.
\newblock In {\em 30th Annual Symposium on Foundations of Computer Science},
  332--337.
\newblock IEEE.

\bibitem[\protect\citeauthoryear{Zhang and Pavone}{2016}]{Zhang2016}
Zhang, R., and Pavone, M.
\newblock 2016.
\newblock {Control of Robotic Mobility-On-Demand Systems: a
  Queueing-Theoretical Perspective}.
\newblock {\em The International Journal of Robotics Research}
  35(1-3):186--203.

\end{thebibliography}

\end{document}